\documentclass[ECP, preprint]{ejpecp} 

\usepackage{float}

\SHORTTITLE{Inference for multiscale fOU processes}

\TITLE{Statistical Inference for the Rough Homogenization Limit of Multiscale Fractional Ornstein-Uhlenbeck Processes} 

\AUTHORS{%
  Pablo Ramses Alonso-Martin\footnote{Deparment of Statistics, University of Warwick, United Kingdom.
    }\footnote{\EMAIL{pablo.alonso-martin@warwick.ac.uk}}\orcid{0009-0005-5206-9522}
    \and 
    Horatio Boedihardjo \footnotemark[1]\footnote{\EMAIL{horatio.boedihardjo@warwick.ac.uk}}
    \and 
  Anastasia Papavasiliou\footnotemark[1]\footnote{\EMAIL{a.papavasiliou@warwick.ac.uk}}\orcid{0000-0002-3861-9621}
    }

\KEYWORDS{multiscale fractional Ornstein-Uhlenbeck; spectral norm; fractional Brownian Motion; fractional Gaussian noise} 

\AMSSUBJ{60H30; 62F12; 60H05; 62M09; 60G22} 

\SUBMITTED{October 15, 2024} 
\ACCEPTED{March 12, 2025} 

\VOLUME{30}
\YEAR{2025}
\PAPERNUM{29}
\DOI{10.1214/25-ECP677}

\ABSTRACT{We study the problem of parameter estimation for the homogenization limit of multiscale systems involving fractional dynamics. In the case of stochastic multiscale systems driven by Brownian motion, it has been shown that in order for the Maximum Likelihood Estimators of the parameters of the limiting dynamics to be consistent, data needs to be subsampled at an appropriate rate. We extend these results to a class of fractional multiscale systems, often described as scaled fractional kinetic Brownian motions. We provide convergence results for the MLE of the diffusion coefficient of the limiting dynamics, computed using multiscale data. This requires the development of a different methodology to that used in the standard Brownian motion case, which is based on controlling the spectral norm of the inverse covariance matrix of a discretized fractional Gaussian noise on an interval.}

\begin{document}


\section{Introduction}


Real-world systems often exhibit a multiscale behavior. A typical example of that is the weather, where atmospheric dynamics take place at the scale of seconds but it is only their aggregate effect that determines their long-term dynamics \cite{climate}. To study their behavior, models are typically designed using different spatial or temporal units in each component. However, these models become difficult to work with as it is not feasible to observe the fast varying variables. Models are then simplified taking appropriate limits that allow to disregard the fast varying components explicitly while taking into consideration the way in which they alter the long term dynamics of the system. This leads to coarse-grained models that allow for a much better understanding of the desired behaviors of the system. When it comes to systems of Stochastic Differential Equations (SDEs), these are the well-known averaging and homogenization equations \cite{multiscalebook}.

The question that naturally arises in this scenario is whether it is possible to fit the coarse-grained model to data coming from the full multiscale model, which is the one to be observed. This problem was first studied in \cite{stupav} for a particular class of homogenised SDE, where they showed that the Maximum Likelihood Estimators (MLE) for the effective drift and diffusion coefficients are asymptotically biased when evaluated at multiscaled data unless the data is subsampled at an appropriate rate. More specifically, they showed that if observations from the multiscale process are sampled at rate $\delta=\epsilon^\alpha$ then the MLE for the effective drift and diffusion coefficients evaluated on a discretized path $\{X^{\epsilon}_{i\delta}\}_{i=0}^{N}$ are consistent if $0<\alpha<1$, i.e. if the sampling rate is between the two characteristic time-scales of the system. In \cite{tessystuart}, the results of \cite{stupav} for the drift estimation are extended to more general multiscale models. In particular, the authors prove that the MLE for the drift of the coarse-grained equation evaluated at multiscale data is asymptotically unbiased in the averaging case whereas for the homogenised equation it is biased unless the sampling is again performed at an $\epsilon$-dependent rate of $\delta=\epsilon^\alpha$ for $\alpha\in(0,1)$. This is not at all unexpected as the appearance of higher frequency terms leads to the multiscale dynamics dominating in the finer scales. In \cite{tessyzhang}, by focusing on the particular case of a multiscale system of Ornstein-Uhlenbeck processes, they show that the MLEs for the drift and diffusion coefficients will be asymptotically unbiased for $\alpha\in(0,\frac{1}{2})$ and $\alpha\in(0,1)$ respectively, with optimal sampling rate achieved at $\alpha=\frac{1}{4}$ for the drift and $\alpha=\frac{2}{3}$ for the diffusion coefficient. They also show that in the averaging case, the MLEs for both the drift and diffusion coefficients are unbiased.

Recently, there has been significant progress in understanding averaging and homogenisation limits for multiscale systems driven by fractional Brownian motion, which naturally led to considering statistical inference problems in this context. In \cite{discretetimeinf}, the authors study statistical inference questions for the averaged limit of multiscale systems that are perturbed by a fractional Brownian motion as described in \cite{perturbedfbm}. The parameters of interest in \cite{discretetimeinf} describe the behavior of a deterministic averaged system. In this paper, we are instead concerned with inference on models that exhibit a behavior in the high-frequency regime resembling that of a homogenized equation. Given the random nature of the latter, as opposed to the deterministic of the former, the estimation techniques will differ. However, we are able to adapt some of the results in Section 3 in \cite{discretetimeinf} to estimate $H$ when the two time scales are sufficiently separated.  

More precisely, we consider as fast components a fractional Ornstein-Uhlenbeck process for which a 'rough homogenisation' limit exists and is a scaled fractional Brownian motion \cite{ximeli}. This corresponds to a fractional kinetic Brownian motion model which is also found in the literature as a generalization of the well-known physical Brownian motion model \cite{physicalbm, kslimits}. Similar to \cite{stupav}, we study the behavior of the MLE for the diffusion parameter $\sigma$ of the limiting equation when evaluated on multiscale data. We assume access to continuous paths $X^{\epsilon}_{t}(\omega)$ of the multiscale system \eqref{eq: kfbm} and we consider the behavior of estimators given data sampled at any constant rate $\delta >0$. As before, we find that for the estimator to be unbiased, we need to sample $X^{\epsilon}$ at an $\epsilon$-dependent rate $\delta$ in between the two time-scales of the system \eqref{eq: kfbm}: $\epsilon<\delta<1$. In particular, we show that the MLE for $\sigma$ vanishes to $0$ almost surely if no subsampling is performed. On the other hand, we show that if subsampling is performed at an appropriate rate $\delta = \epsilon^\alpha$ convergence to $\sigma$ for $\alpha$ in an appropriate interval is ensured in the $L^2$-norm. We will assume that the observational time horizon $T$ is arbitrary but fixed to a finite quantity. 

This model requires the development of a new methodology to address the problem. Existing results in the literature on parameter estimation for multiscale systems heavily rely on the Markovianity of the processes involved. Here, the fast and slow processes in the multiscale equations, as well as the effective limit, depart from the Markovian regime. Thus, instead of using the generators of the processes to draw convergence results of the estimators, we derive convergence rates for the covariance matrix of the increments in the slow processes with respect to the spectral norm and the size of the discretisation, to measure the amplification of the error picked up when using multiscale data.
\section{Preliminaries}
\begin{definition}
A \textbf{fractional Brownian motion} (fBM) $(B^{H}_t)_{t\geq 0}$ of Hurst index $H\in(0,1)$ is a continuous and centred Gaussian process with covariance function 
\begin{equation}
\label{covariance}
\mathbb{E}\left[B^{H}_tB^{H}_s\right] = \frac{1}{2}(t^{2H}+s^{2H}-|t-s|^{2H})
\end{equation}
\end{definition}
\begin{definition}
A \textbf{fractional Ornstein-Uhlenbeck} process is defined as the unique stationary solution to the Langevin equation
\begin{equation}
\label{fOUeq}
dY_t =-\lambda Y_t dt + \sigma dB^H_t, \quad Y_0=\sigma\int_{-\infty}^0e^{\lambda s}dB_s^H
\end{equation}
which admits the closed formula 
$$
Y_t =e^{-\lambda t}\left(\sigma\int_{-\infty}^te^{\lambda s}dB^H_s\right)
$$
\end{definition}
\begin{definition}
We define the \textbf{kinetic fractional Brownian Motion}, $(X^{\epsilon}_t)_{t\geq 0}$, for some fixed $\epsilon >0$, to be the solution of the system 
\begin{equation}
\label{eq: kfbm}
\left\{
    \begin{array}{lr}
        dX_t^{\epsilon} = \sigma_1\epsilon^{H-1}Y_t^{\epsilon}dt, & X_0^{\epsilon} = x_0\\
        dY_t^{\epsilon} = -\frac{1}{\epsilon}Y_t^{\epsilon}dt + \frac{\sigma_2}{\epsilon^{H}}dB_t^H, & Y^{\epsilon}_0=\epsilon^{-H}\sigma_2\int_{-\infty}^{0}e^{\frac{t-s}{\epsilon}}dB^{H}_{s}
    \end{array}
\right. 
\end{equation}
\end{definition}
\begin{proposition}[\cite{ximeli}]
Let $X^{\epsilon}_t$ be the solution of \eqref{eq: kfbm}. For any $H\in(0,1)$, $p>1$ and $T>0$ it holds that 
$$\sup_{s,t\in[0,T]}\left|\left|X_{s,t}^{\epsilon} - \overline{\sigma} B^H_{s,t}  \right|\right|_{L^p}\lesssim \epsilon^{H}$$
where $\lesssim$ denotes less or equal up to a positive constant and $\overline{\sigma}$ is the homogenized diffusion coefficient $\overline{\sigma} = \sigma_1\cdot\sigma_2$. 
\end{proposition}

\subsection{Estimators for $\overline{\sigma}$}
We want to derive appropriate estimators for $\overline{\sigma}$ in \eqref{eq: kfbm} based on discretized samples $\{X_{i\delta}^{\epsilon}\}_{i=0}^{N}$ from the process $X^{\epsilon}_t$ defined as the solution to \eqref{eq: kfbm} when this is observed over fixed time horizon $T$ with grid size $\delta=T/N >0$ chosen in advance. 

Since $X_t^{\epsilon}\stackrel{L_p}{\longrightarrow}X_t = \overline{\sigma} B_t^H$, we use the properties of fBM to derive a desired estimator. The fact that it is Gaussian and has stationary increments \cite{fbmbook} allows us to derive an estimator for the diffusion coefficient.

We will use the notation $\Delta_{\delta}Z = \left(Z_{(i+1)\delta}-Z_{i\delta}\right)_{i=1}^N$ for the vector of increments of any process $Z_t$. The covariance between any pair $\left(\Delta_{\delta}B^{H}\right)_{i}$ and $\left(\Delta_{\delta}B^{H}\right)_{j}$ is 
$$P_{ij} = \frac{1}{2}\delta^{2H}\left[(|i-j|+1)^{2H}+(|i-j|-1)^{2H}-2|i-j|^{2H}\right].$$
Consequently, $\Delta_{\delta}X$ follows a normal distribution with $0$ mean and $\sigma^2 P$ variance and it follows by a straightforward calculation that the maximum likelihood estimator for $\sigma^2$ given a realization of $\Delta_{\delta}X$ is 
\begin{equation*}
\hat{\sigma}^{2} = \frac{1}{N}\left(\Delta_{\delta}X\right)^{T}P^{-1}\Delta_{\delta}X
\end{equation*}
where $T$, as a superscript of a vector or matrix denotes transpose. Replacing $X$ by $X^\epsilon$ leads to the following estimtor:
\begin{equation}
\label{MLE_eps}
\hat{\sigma}^{2}_{\delta,\epsilon} = \frac{1}{N}\left(\Delta_{\delta}X^{\epsilon}\right)^{T}P^{-1}\Delta_{\delta}X^{\epsilon}
\end{equation}
The remainder of the paper focuses on studying the behavior of this estimator.

\section{Main results}
In this section we present the main results of our work pertaining to the convergence of the estimator $\hat{\sigma}^{2}_{\delta,\epsilon}$ alongside some newly developed tools necessary for the proofs.

\subsection{Asymptotic biasedness without subsampling}
First, we show that if we sample too finely (in a sense made precise below), the estimator will be biased.
\begin{theorem}
\label{noconvergence}
Let $X^{\epsilon}$ be the solution of \eqref{eq: kfbm} and fix $T=N\delta$ for any time horizon $T$ and $0<H<1$. For any fixed $\epsilon>0$ it holds that 
\begin{equation}
\label{unbiasedno}
\lim_{\delta\rightarrow 0}\hat{\sigma}^{2}_{\delta, \epsilon} = 0 \quad \text{a.s.},
\end{equation}
where $\sigma_{\delta,\epsilon}$ is defined in \eqref{MLE_eps}. Moreover, if $\delta(\epsilon)>0$ is such that $\delta = \epsilon^{\alpha}$ for some $\alpha>0$ with $\alpha>\max\{1,2(1-H)\}$ then
\begin{equation}
\label{unbiasedia}
\mathbb{E}[\hat{\sigma}^{2}_{\delta,\epsilon}]\xrightarrow{\epsilon\rightarrow 0^+}0
\end{equation}
\end{theorem}
In essence, \eqref{unbiasedno} highlights the need of some subsampling as taking $\delta$ arbitrarily small  makes the estimator vanish almost surely. Moreover, \eqref{unbiasedia} means that the estimator will not be consistent for $\alpha >1$ in the case $H\geq 1/2$ or $\alpha>2(1-H)$ when $H<1/2$. 

\subsection{Consistency with appropriate subsampling}
The next result gives the range of $\alpha$ for which we can guarantee that the estimator $\hat{\sigma}_{\delta,\epsilon}$ will be consistent, in a sense that is made precise below.
\begin{theorem}
\label{maintheorem}
Let $\hat{\sigma}_{\delta,\epsilon}^2$ be the estimator defined in \eqref{MLE_eps} constructed from the solution of \eqref{eq: kfbm}, $X^{\epsilon}_t$, observed at sampling rate $\delta = \epsilon^{\alpha}$ where $\epsilon>0$ is the scale separation parameter in \eqref{eq: kfbm}. For any $0<H<1$ and $0<\alpha<\min\{1, \frac{H}{1-H}\}$ it holds that $$\hat{\sigma}_{\delta,\epsilon}^2\xrightarrow{\epsilon\rightarrow 0^+}\overline{\sigma}^2\quad\text{in }L^2$$ 
\end{theorem}

Note that in the case $H\geq 1/2$, the estimator converges if and only if $0<\alpha<1$, with Theorem \ref{noconvergence} ensuring that the threshold $\alpha = 1$ is sharp. However, in the case $H<1/2$, we can only show consistency for $0<\alpha<\frac{H}{1-H}$ and biasedness for $\alpha>2(1-H)$, with numerical examples suggesting that $\alpha=1$ might be the critical value here as well (see Figure \ref{fig: simulation of error}). 
\begin{figure}
    \centering
    \includegraphics[scale=0.367]{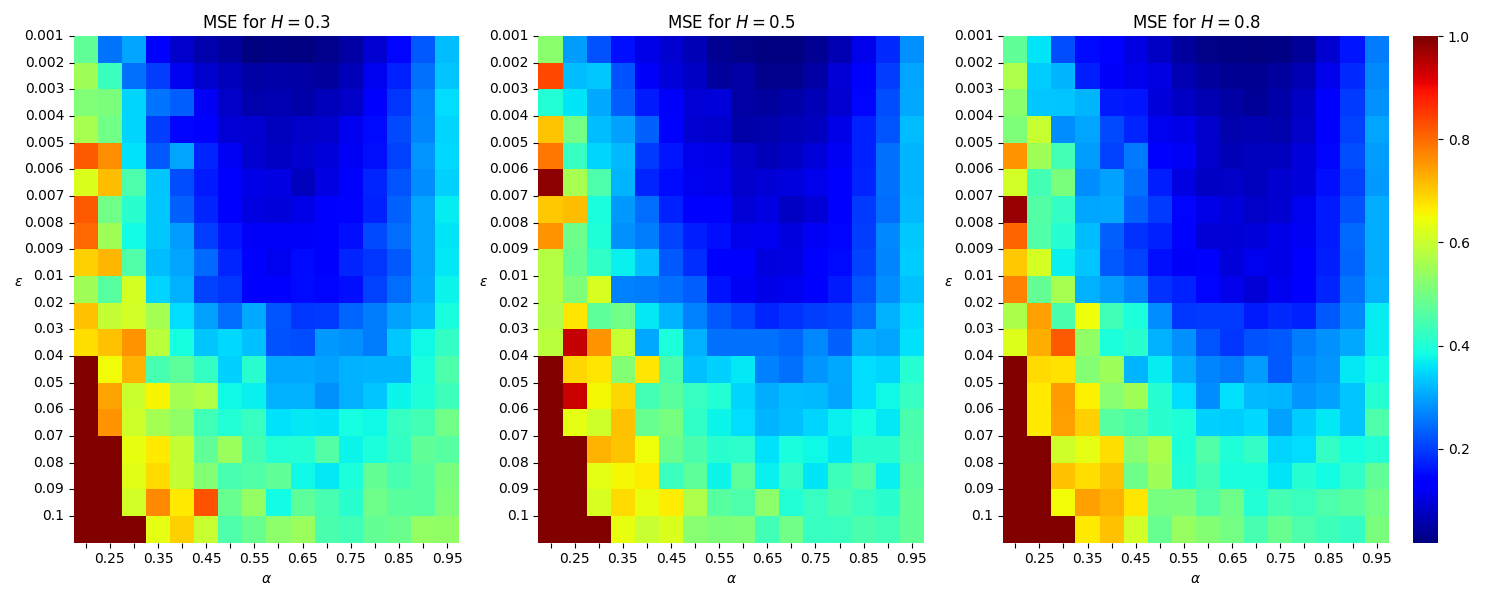}
    \caption{Numerical simulation of the Mean Squared Error in the estimation for different values of $\epsilon$ and $\alpha$ when $\sigma = T = 1$.}
    \label{fig: simulation of error}
\end{figure}                                                                                         

\subsection{Auxiliary Tools: asymptotic bounds on the spectral norm of $P^{-1}$.}
A key step in our approach is controlling the spectral norm of the inverse covariance matrix of fractional Gaussian Noise  (fGN) on a equi-spaced grid. While this is a trivial problem in the case of Gaussian noise arising from a standard Brownian Motion as covariance matrices are diagonal, it gets quite complicated when $H\neq 1/2$. The problem of bounding the eigenvalues has been studied for covariance operators associated with both fBM and fGN in \cite{chigansky18}, where explicit asymptotic results were derived for the eigenvalues of the operators both for $H<1/2$ and $H>1/2$. However, for the discretized processes, there is not, to the best of our knowledge, anything similar available in the literature.

\begin{lemma}
\label{orderlemma}
Let $P$ be the covariance matrix of a discretized fGN process on an equi-spaced grid of $n$ points over an interval $[0,T]$. It then holds that
$$||P^{-1}||_2 \leq Cn^{\beta}$$
where $\beta = \max\{1,2H\}$.
\end{lemma}

Note that there is a qualitative difference in the behavior of the spectral norm in lemma \ref{orderlemma}, for $H>1/2$ and $H<1/2$. When $H>1/2$ the spectral norm of the matrix is asymptotically bounded by $n^{2H}$. However, when $H<1/2$ it stops decreasing with $H$ and does not get any better than $n$.
\section{Proof of Lemma \ref{orderlemma}}
\label{section: proof of main lemma}
Throughout this section we will assume that $T=1$ without any loss of generality due to the self-similarity property of fBM \cite{fbmbook}. We also add the superscript $(n)$ to the covariance matrix to make explicit the dependence of $||P^{-1}||_2$ on the size of the grid. 

It is a standard linear algebra result that the spectral norm of the inverse covariance matrix is bounded by the reciprocal of the smallest eigenvalue. This allows to reformulate the problem as finding a lower bound for the quadratic form associated to $P^{(n)}$ of the form 
\begin{equation*}
u^T P^{(n)} u \geq C \sum_{i=1}^n u_i^2
\end{equation*}
for any eigenvector $u\in\mathbb{R}^{n}$ associated to $P^{(n)}$ as $C^{-1}$ would automatically be an upper bound for the desired quantity.

Finally, we note that the quadratic form associated with the covariance matrix $P^{(n)}$ of fGN can also be expressed as the second moment of a Wiener integral with respect to the original fBM as follows:
\begin{align}
\label{eq: qudratic form as expectation}
\nonumber u^T P^{(n)} u &= \sum_{i,j = 0}^{n-1}u_iu_j P^{(n)}_{i,j}
 = \sum_{i,j = 0}^{n-1}u_iu_j \mathbb{E}[(B^{H}_{(i+1)\delta} - B^{H}_{i\delta})(B^{H}_{(j+1)\delta} - B^{H}_{j\delta})]\\
\nonumber &= \mathbb{E}\left[\left(\sum_{i=0}^{n-1}u_i(B^{H}_{(i+1)\delta} - B^{H}_{i\delta})\right)^{2}\right]
=\mathbb{E}\left[\left(\int_0^1f(t)dB^{H}_t \right)^2\right]
\end{align}
where $f$ is the simple function $\sum_{i=0}^{n-1}u_i\mathbf{1}_{[i\delta, (i+1)\delta)}$.

\subsection{The Wiener Integral with respect to fractional Brownian Motion}
Before proving lemma \ref{orderlemma}, we first define the space $\mathcal{H}$ of deterministic integrands of a fBM, following \cite{MalliavinNualart}. Let $\mathcal{E}$ be the space of $\mathbb{R}$-valued step functions and denote by $\mathcal{H}$ the closure of $\mathcal{E}$ for the scalar product 
$$\langle\mathbf{1}_{[0,t)},\mathbf{1}_{[0,s)}\rangle_{\mathcal{H}} = \mathbb{E}[B_t^H B_s^H]$$
which is an equivalent definition to the natural construction of the Wiener integrands as the span of the increments of $B^H_t$ (see, for example, \cite{jolis}). This is to say that, for $\phi,\ \psi\in\mathcal{H}$, the inner product admits the representation 
$$\langle \phi,\psi\rangle_{\mathcal{H}} = \mathbb{E}\left[\int_{\mathbb{R}}\phi(t)dB^H_t \cdot\int_{\mathbb{R}}\psi(t)dB^H_t \right]$$

The proof of lemma \ref{orderlemma}, although based on the same idea of studying the integrands of the Wiener integral with respect to a fBM, is substantially different depending on whether $H$ is smaller or bigger than $1/2$. Therefore, we will carry an independent analysis for each case. 

Note that the functions we will be interested in are right-continuous step functions defined in
 $[0,1]$ which we will denote as $\mathcal{E}_{|[0,1]}$ and we will consider them extended to the real line by taking $f(x)=0$ when $x\notin [0,1]$ when needed so that $\mathcal{E}_{|[0,1]}\subset \mathcal{E}$. 

Before proceeding with each case note that in principle none of them would apply to the case $H=1/2$. However, in this case the results is trivial as $P^{(n)} = \delta I_{n}$ and therefore $||(P^{(n)})^{-1}||_2 = n$. 

\subsection{Case $H>1/2$.}
We begin by surveying some concepts from fractional calculus alongside some auxiliary results as these are very closely related to the study of Wiener integrals against fBM \cite{taqqu}. 

\begin{definition}
The fractional integrals of order $\gamma$ are defined as the operator 
$$\mathcal{I}^{\gamma}_{\pm}h(t) = \frac{1}{\Gamma(\gamma)}\int_{\mathbb{R}}(t-s)_{\pm}^{\gamma -1}h(s)ds$$
whereas the Marchaud fractional derivatives of order $\gamma$ are defined as
\begin{equation}
\label{eq: marchaud fractional derivative}
\mathcal{D}^{\gamma}_{\pm}h(t) = \frac{\gamma}{\Gamma(1-\gamma)}\int^{\infty}_{0}\frac{h(t)-h(t\mp s)}{s^{\gamma + 1}}ds.
\end{equation}
\end{definition}  
\begin{remark}
We would like to stress that Marchaud fractional derivatives are typically found in the literature defined conveniently as a limit truncating the integral inside \eqref{eq: marchaud fractional derivative} to include the case in which it converges only conditionally. For our purposes \eqref{eq: marchaud fractional derivative} is good enough.
\end{remark}
Under various assumptions, $\mathcal{D}_{\pm}^{\gamma}$ is the inverse operator of $\mathcal{I}^{\gamma}_{\pm}$ in $L^{2}(\mathbb{R}^{+})$ whereas both the fractional integral and derivatives with different signs are adjoint to each other --see \cite{fractionalcalculus}. 

When $H>1/2$ it is possible to find a space of functions $\Lambda^{H}$ such that $\mathcal{E}\subset\Lambda^{H}\subset \mathcal{H}$ satisfying
\begin{equation}
\label{eq:H-norm and L2-norm equivalence}
||f||_{\mathcal{H}} = c_H||\mathcal{I}^{H-1/2}f||_{L^2(\mathbb{R}^+)}
\end{equation}
for any $f\in\Lambda^{H}$ where $c_H$ is a positive constant depending only on $H$ --see \cite{taqqu}. 

\begin{proposition}
\label{prop: norm bound}
Let $H>1/2$, $f\in \mathcal{E}_{|[0,1]}$. It holds that:
\begin{equation}
\label{eq: norm bound}
||f||_{L^2(\mathbb{R})}^2\leq c_H||f||_{\mathcal{H}}||\mathcal{D}^{H-1/2}_{-}f||_{L^2([0,1])}.
\end{equation}
\end{proposition}
\begin{proof}
A simple application of Cauchy-Schwardz's inequality yields 
\begin{align*}
|\langle f,f \rangle_{L^2(\mathbb{R})} | &= |\langle \mathcal{D}_{+}^{H-1/2}\mathcal{I}^{H-1/2}_{+}f,f \rangle_{L^2(\mathbb{R})} | = |\langle \mathcal{I}^{H-1/2}_{+}f,\mathcal{D}^{H-1/2}_{-}f \rangle_{L^2([0,1])}|\\
&\leq ||\mathcal{I}^{H-1/2}_{+}f||_{L^2(\mathbb{R}^+)}\cdot ||\mathcal{D}^{H-1/2}_{-}f||_{L^2([0,1])} = c_H||f||_{\mathcal{H}}\cdot ||\mathcal{D}^{H-1/2}_{-}f||_{L^2([0,1])}
\end{align*}
where the last equality follows from \eqref{eq:H-norm and L2-norm equivalence}. The the fact that $\mathcal{E}_{|[0,1]}\subset L^{p}(\mathbb{R})$ for any $1\leq p\leq \infty$ guarantees that the first two equalities hold. The change in the domain from $\mathbb{R}$ to $[0,1]$ in the second equality is due to the fact that $\mathcal{I}^{\gamma}_{+}f$ is supported on $(0,+\infty)$ and $\mathcal{D}^{\gamma}_{-}f$ in $(-\infty,1)$ when $f$ is supported in $[0,1)$.
\end{proof}
Using the previous result we can lower-bound the quadratic form associated to $P^{(n)}$ by taking again $f=\sum_{i=0}^{n-1}u_i\mathbf{1}_{[i\delta, (i+1)\delta)}$
\begin{equation}
\label{quadraticbound}
u^T P^{(n)}u = ||f||_{\mathcal{H}}^2 \geq \left(\frac{||f||_{L^2(\mathbb{R}^+)}^2}{\sqrt{c_H}||\mathcal{D}^{H-1/2}_{-}f||_{L^2(\mathbb{R}^+)}}\right)^2.
\end{equation}
Hence, if we can find a bound of the type 
\begin{equation}
\label{eq: another bound}
    \left(\frac{||f||_{L^2(\mathbb{R}^+)}^2}{\sqrt{c_H}||\mathcal{D}^{H-1/2}_{-}f||_{L^2(\mathbb{R}^+)}}\right)^2 \geq C(n)\sum_{i=0}^{n-1}u_i^2,
\end{equation}
it is possible to write
\begin{equation}
\label{boundinversecov}
||(P^{n})^{-1}||_{2}\leq C(n)^{-1}.
\end{equation}

We now prove that \eqref{eq: another bound} holds for an appropriate $C(n)$, to be identified. For $f=\sum_{i=0}^{n-1}u_i\mathbf{1}_{[i\delta, (i+1)\delta)}$ the following holds: 
\begin{align*}
||f||_{L^2(\mathbb{R}^+)}^2 =
\int_0^1 \left(\sum_{i=0}^{n-1}u_i\mathbf{1}_{[\delta, (i+1)\delta)}(s)\right)^2 ds= \int_0^1 \sum_{i=0}^{n-1}u_i^2\mathbf{1}_{[\delta, (i+1)\delta)}(s) ds
= \delta \sum_{i=0}^{n-1}u_i^2.
\end{align*}
The Marchaud fractional derivative of order $\gamma$ of $f$ may be written as 
\begin{align*}
\mathcal{D}_{-}^{\gamma} (f)(t)=&\frac{-\gamma}{\Gamma(1-\gamma)}
\sum_{i=0}^{n-1}\left(\sum_{j = i+1}^{n-1}(u_j - u_i)\left(((j+1)\delta - s)^{-\gamma} - (j\delta - s)^{\delta}\right)\right) \mathbf{1}_{[i\delta, (i+1)\delta)}(t)
\end{align*}
when $t>0$. It follows that
\begin{align}
&||\mathcal{D}_{-}^{\gamma}(f)(t)||_{L^2(\mathbb{R}^+)}^{2}
=\sum_{i=0}^{n-1}||\mathcal{D}_{-}^{\gamma}(f)(t)||_{L^2([i\delta, (i+1)\delta))}^{2} \nonumber \\
=& C \sum_{i=0}^{n-1}\int_{i\delta}^{(i+1)\delta}\left(\sum_{j\geq i+1}(u_j-u_i)\left(((j+1)\delta -r)^{-\gamma} - (j\delta-r)^{-\gamma}\right)\right)^2dr\nonumber \\
=&C\delta^{1-2\gamma} \sum_{i=0}^{n-1}\int_{i}^{i+1}\left(\sum_{j\geq i+1}(u_j-u_i)\left((j+1 -s)^{-\gamma} - (j-s)^{-\gamma}\right)\right)^2ds \label{line: cv}\\
=&C \delta^{1-2\gamma} \sum_{i=0}^{n-1}\sum_{j_1,j_2 \geq i+1}(u_{j_1}-u_i)(u_{j_2}-u_i) \nonumber\\
&\int_{i}^{i+1}\left((j_1+1 -s)^{-\gamma} - (j_1-s)^{-\gamma}\right)\left((j_2+1 -s)^{-\gamma} - (j_2-s)^{-\gamma}\right)ds\label{line: expansion}
\end{align}
where in \eqref{line: cv} we take the change of variable $s = r/\delta$ and in \eqref{line: expansion} we expand the square and pull out of the integral the terms without $s$. We denote
\begin{align*}
I(i,j_1,j_2) =
&\int_{i}^{i+1}\left((j_1+1 -s)^{-\gamma} - (j_1-s)^{-\gamma}\right)\left((j_2+1 -s)^{-\gamma} - (j_2-s)^{-\gamma}\right)ds.
\end{align*}
We want to find a bound of the form $C\sum_{i=0}^{n-1}u_i^2$ that will allow to conclude the result using \eqref{boundinversecov}. Expanding the product of differences, using $|ab|\leq \frac{1}{2}(a^2 + b^2)$ and the fact that $I(i,j_1,j_2)\geq 0$, we get the bound 
\begin{align}
||\mathcal{D}_{-}^{\gamma}(f)(t)||_{L^2(\mathbb{R}^+)}^{2}&\leq  C\delta^{1-2\gamma} \sum_{i=0}^{n-1}\sum_{j_1,j_2 \geq i+1}(2u_i^2 + u_{j_1}^2 + u_{j_2}^2)I(i,j_1,j_2) \nonumber\\ 
 &=C\delta^{1-2\gamma} \sum_{i=0}^{n-1}\sum_{j_1,j_2 \geq i+1}(2u_i^2 + 2u_{j_2}^2)I(i,j_1,j_2) 
\label{eq: bound for adjoint norm}
\end{align}
where the last equality is due to the role of $j_1$ and $j_2$ being interchangeable in $I(i, j_1, j_2)$. Thus, we only need to bound the terms
$$
A_i = \sum_{i=0}^{n-1}\sum_{j_1,j_2 \geq i+1}u_i^2 I(i,j_1,j_2), \quad
B_{j_2} = \sum_{i=0}^{n-1}\sum_{j_1,j_2 \geq i+1}u_{j_2}^2I(i,j_1,j_2)
$$

After some tedious computations (included in the Supplementary Material as they are just a matter of calculations and do not add anything to the proof) it is possible to bound all the terms by expressions of the form $C\sum_{i=0}^{n-1}u_i^2$ when $0<\gamma<1/2$, which is true for the case of interest $\gamma = H-1/2$ when $1/2<H<1$.

Plugging this in \eqref{eq: bound for adjoint norm}, we conclude that $\mathcal{D}_{-}^{H-1/2}(f)\in L^2(\mathbb{R}^+)$. Moreover, 
\begin{equation}
||\mathcal{D}_{-}^{H-1/2}(f)(t)||_{L^2(\mathbb{R}^+)}^{2}\leq C\delta^{1-2\gamma}\sum_{i=0}^{n-1}u_i^2.
\end{equation}
Hence, substituting these in \eqref{quadraticbound}, we get that 
\begin{equation*}
u^TP^{n}u \geq \left(\frac{\delta\sum_{i=0}^{n-1}u_i^2}{C\delta^{1/2-\gamma}\left(\sum_{i=0}^{n-1}u_i^2\right)^{1/2}}\right)^2 = C\delta^{1+2\gamma}\sum_{i=0}^{n-1}u_i^2.
\end{equation*}
The result follows from \eqref{boundinversecov}, as 
\begin{equation}
\label{orderlarge}  
||(P^{n})^{-1}||_2 \leq C\delta^{-1-2\gamma} = C\delta^{-2H} = Cn^{2H}
\end{equation}
 
\subsection{Case $H<1/2$.}
\label{subsec: H<1/2}
As before, let $\mathcal{H}$ be the domain of the Wiener integral with respect to fBM \cite{fbmbook}. To construct a lower bound for the norm defined on this space, we can now take advantage of one of the multiple representations this space has, which will in turn allow us to write 
$$||f||_{\mathcal{H}}\geq C ||f||_{L^{2}([0,1])}.$$

First, we note that, for $f\in\mathcal{H}$ (and $H<1/2$), it is possible to write (see \cite{bardina} Theorem 2.5) 
\begin{equation}
\label{eq: rewrite the norm for H<1/2}
||f||_{\mathcal{H}}^{2} = \frac{1}{2}H(1-2H)\int\int_{\mathbb{R}^2}\frac{(f(x) - f(y))^2}{|x-y|^{2-2H}}dxdy.
\end{equation}
 
We now use the fact that (see \cite{hitchhiker} Theorem 6.5) for any compactly supported function $f$ defined in the real line 
\begin{equation}
\label{eq: lower bound on the sobolev seminorm}
\int\int_{\mathbb{R}^2}\frac{(f(x) - f(y))^2}{|x-y|^{2-2H}}dxdy\geq C||f||_{L^{q}([0,1])}
\end{equation}
for any $q\in[2,1/H]$. 

The result is concluded by taking $f=\sum_{i=0}^{n-1}u_i\mathbf{1}_{[i\delta, (i+1)\delta)}$ and putting \eqref{eq: rewrite the norm for H<1/2} and \eqref{eq: lower bound on the sobolev seminorm} together
$$||f||_{\mathcal{H}}^2\geq C||f||_{L^{2}(\mathbb{R})}^2=C||f||_{L^{2}([0,1])}^2 = C\delta\sum_{i=1}^n u_i^2,$$ 
which is equivalent to 
$$||P^{(n)}||_2\geq C\delta.$$ 
We conclude that, if $H<1/2$ 
\begin{equation}
\label{ordersmall}
||(P^{(n)})^{-1}||_2\leq C\delta^{-1}.
\end{equation}
Lemma \ref{orderlemma} is a combination of \eqref{orderlarge} and \eqref{ordersmall} for cases $H>\frac{1}{2}$ and $H<\frac{1}{2}$ respectively.
\section{Proof of Theorem \ref{noconvergence}}
\begin{proof}
Recall that $\left(X_t^\epsilon\right)_{t\geq 0}$ and $\left(Y_t^\epsilon \right)_{t\in\mathbb{R}}$ are both defined in the same space $(\Omega, \mathcal{F}, \mathbb{P})$. By means of Kolmogorov's Continuity Theorem (see for example \cite{yorbook}) there exists a set $A_{T,\epsilon}\in\mathcal{F}$ with $\mathbb{P}(A_{T,\epsilon})=1$ in which, for every $\omega\in A_{T,\epsilon}$, $Y^{\epsilon}_{t}(\omega)$ is a continuous function on $t$ in $[0,T]$. As a consequence, for each $\omega\in A_{T,\epsilon}$, there exists $C_{T,\epsilon}(\omega)$ such that 
$$\max_{t\in[0,T]}|Y^{\epsilon}_t| < C_{T,\epsilon}(\omega).$$
Let now $\omega\in A_{T,\epsilon}$
\begin{equation}
\label{firstbound}
\hat{\sigma}_{\delta, \epsilon}^{2}(\omega) \leq ||P^{-1}||_{2}\frac{1}{N}\sum_{i=0}^{N-1}\left(X^{\epsilon}_{(i+1)\delta}(\omega) -  X^{\epsilon}_{i\delta}(\omega)\right)^{2}
\end{equation}
It follows from the definition of $X^{\epsilon}$ in \eqref{eq: kfbm} that
\begin{align}
0\leq (X^{\epsilon}_{(i+1)\delta}(\omega) - X^{\epsilon}_{i\delta}(\omega))^{2} &= \sigma_1^2\epsilon^{2H-2} \int_{i\delta}^{(i+1)\delta}\int_{i\delta}^{(i+1)\delta}Y^{\epsilon}_s(\omega)Y^{\epsilon}_u(\omega) dsdu \nonumber\\
&\leq \sigma_1^2\epsilon^{2H-2}\delta^2 \max_{t,s\in[0,T]}Y^{\epsilon}_t(\omega) Y^{\epsilon}_s(\omega) \nonumber\\
&\leq \sigma_1^2\epsilon^{2H-2}\delta^2 \left(\max_{t\in[0,T]}|Y^{\epsilon}_t(\omega)|\right)^2
\end{align}
Putting all together and applying lemma \ref{orderlemma}, we get that 
\begin{equation}
\hat{\sigma}_{\delta, \epsilon}^{2}(\omega)\lesssim \epsilon^{2H-2}\delta^{2-2H\vee 1} \left(\max_{t\in[0,T]}|Y^{\epsilon}_t(\omega)|\right)^2 = C_{T,\epsilon}(\omega)^2\delta^{2-2H\vee 1}.
\end{equation}
Then, \eqref{unbiasedno} follows for any $0<H<1$ and $\epsilon>0$ by taking $\delta\to 0$. To prove \eqref{unbiasedia}, recall that the increments of $X^{\epsilon}$ are identically distributed. An application of Fubini's theorem yields
\begin{align*}
\mathbb{E}[(X^{\epsilon}_{\delta} - X_0^{\epsilon})^2] &= \sigma_1^2\epsilon^{2H-2} \int_{0}^{\delta}\int_{0}^{\delta}\mathbb{E}[Y^{\epsilon}_sY^{\epsilon}_t]dsdt
\leq \sigma_1^2\frac{\delta^2}{\epsilon^{2-2H}}\mathbb{E}[(Y^{\epsilon}_0)^2] = \frac{\delta^2}{\epsilon^{2-2H}}\overline{\sigma}^2\Gamma(2H+1)
\end{align*}
where the inequality is due to Cauchy-Schwarz's inequality together with the stationarity of $Y^{\epsilon}_t$. Hence, taking expectation in both sides of \eqref{firstbound}, we get  
$$\mathbb{E}[\hat{\sigma}^{2}_{\delta,\epsilon}]\leq C\frac{\delta^2}{\epsilon^{2-2H}}||P^{-1}||_2 = C\frac{\delta^{2-2H\vee 1}}{\epsilon^{2-2H}}. $$
 Note that, we now assume that $\delta = \epsilon^\alpha$. Therefore, if $0<H<1/2$
$$\mathbb{E}[\hat{\sigma}^{2}_{\delta,\epsilon}]\leq C\frac{\delta}{\epsilon^{2-2H}} = C\epsilon^{\alpha - 2(1-H)}$$
whereas if $1/2\leq H<1$ 
$$\mathbb{E}[\hat{\sigma}^{2}_{\delta,\epsilon}]\leq C\frac{\delta^{2-2H}}{\epsilon^{2-2H}} = C\epsilon^{(2-2H)(\alpha - 1)}.$$
The result follows by taking $\epsilon\to 0$ in both sides.
\end{proof}

\section{Proof of Theorem \ref{maintheorem}}
\label{section: proof of convergence}

\begin{proof}
We begin by decomposing and applying the $L^2$-norm to the estimator as follows 
\begin{align}
\label{eq: estimator decomposition}
\left|\left|\hat{\sigma}^2_{\epsilon,\delta} -\overline{\sigma}^2\right|\right|_{L^2} \leq& \frac{1}{N}\left|\left|\left(\Delta_{\delta}X^{\epsilon} - \overline{\sigma} \Delta_{\delta}B^{H}\right)^T P^{-1}\left(\Delta_{\delta}X^{\epsilon} - \overline{\sigma} \Delta_{\delta}B^{H}\right)\right|\right|_{L^2}\nonumber\\
&+ \frac{2\overline{\sigma}}{N}\left|\left|\left(\Delta_{\delta}B^{H}\right)^T P^{-1}\left(\Delta_{\delta}X^{\epsilon} - \overline{\sigma} \Delta_{\delta}B^{H}\right)\right|\right|_{L^2}\nonumber\\
&+\left|\left|\frac{\overline{\sigma}^2}{N}\left(\Delta_{\delta}B^{H}\right)^T P^{-1}\left(\Delta_{\delta}B^{H}\right) - \overline{\sigma}^2\right|\right|_{L^2}.
\end{align}
The differences may be written as $\Delta_{\delta}X^{\epsilon} - \overline{\sigma} \Delta_{\delta}B^{H} = \sigma_1\epsilon^{H}\Delta_{\delta}Y^{\epsilon}$ (see \cite{ximeli}). 
For the first term we have that 
\begin{align}
\label{asbound1}
\nonumber \left(\Delta_{\delta}X^{\epsilon} - \overline{\sigma} \Delta_{\delta}B^{H}\right)^T P^{-1}\left(\Delta_{\delta}X^{\epsilon} - \overline{\sigma} \Delta_{\delta}B^{H}\right) &\leq ||\Delta_{\delta}X^{\epsilon} - \overline{\sigma} \Delta_{\delta}B^{H}||_{2}^2||P^{-1}||_2\\
&= ||\sigma_1\epsilon^{H}\Delta_{\delta}Y^{\epsilon}||_{2}^2||P^{-1}||_2.
\end{align}
Since $\sigma_1^2||P^{-1}||_2$ is a deterministic quantity, we only need to control the following quantity
\begin{align}
\label{roughboundCS1}
\mathbb{E}\left[||\epsilon^{H}\Delta_{\delta}Y^{\epsilon}||_{2}^4\right] &= \epsilon^{4H}\mathbb{E}\left[\left(\sum_{i=0}^{N-1}(Y^{\epsilon}_{i\delta, (i+1)\delta})^2\right)\left(\sum_{j=0}^{N-1}(Y^{\epsilon}_{j\delta, (j+1)\delta})^2\right)\right]\nonumber\\
&=\epsilon^{4H}\sum_{i,j=0}^{N-1}\mathbb{E}\left[(Y^{\epsilon}_{i\delta, (i+1)\delta})^2(Y^{\epsilon}_{j\delta, (j+1)\delta})^2\right]\nonumber\\
&\leq \epsilon^{4H}\sum_{i,j=0}^{N-1}\mathbb{E}[(Y^{\epsilon}_{i\delta, (i+1)\delta})^4]^{1/2}\mathbb{E}[(Y^{\epsilon}_{j\delta, (j+1)\delta})^4]^{1/2}\nonumber\\
&= \epsilon^{4H}N^{2}\mathbb{E}[(Y^{\epsilon}_{0,\delta})^4] = 3\epsilon^{4H}N^{2}\mathbb{E}[(Y^{\epsilon}_{0,\delta})^2]^2
\end{align}
where we used Cauchy-Schwardz's inequality and the stationarity of the fractional Ornstein-Uhlenbeck process $Y^{\epsilon}$. It follows that 
\begin{align*}
\left| \left| ||\epsilon^{H}\Delta_{\delta}Y^{\epsilon}||_{2}^2 \right|\right|_{L^2} = \left(\mathbb{E}\left[||\epsilon^{H}\Delta_{\delta}Y^{\epsilon}||_{2}^4\right]\right)^{1/2}
\lesssim \epsilon^{2H}N.
\end{align*}
Together with Lemma \ref{orderlemma}, this yields
\begin{equation}
\label{term1}
\frac{1}{N}\left|\left|\left(\Delta_{\delta}X^{\epsilon} - \overline{\sigma} \Delta_{\delta}B^{H}\right)^T P^{-1}\left(\Delta_{\delta}X^{\epsilon} - \overline{\sigma} \Delta_{\delta}B^{H}\right)\right|\right|_{L^2} \lesssim \epsilon^{2H}/\delta^{\beta}
\end{equation}
by just plugging each bound in \eqref{asbound1}. We use a similar approach to bound the second term of \eqref{eq: estimator decomposition}. First, we note that 
\begin{equation}
\label{asbound2}
\left|\left|\left(\overline{\sigma} \Delta_{\delta}B^{H}\right)^T P^{-1}\left(\Delta_{\delta}X^{\epsilon} - \overline{\sigma} \Delta_{\delta}B^{H}\right)\right|\right|_{L^2} \leq ||\overline{\sigma}\Delta_{\delta}B^{H}||_{2}||\sigma_1\epsilon^{H}\Delta_{\delta}Y^{\epsilon}||_{2}||P^{-1}||_2.
\end{equation}
Thus, we now need to control $\left| \left| ||\sigma_1\epsilon^{H}\Delta_{\delta}Y^{\epsilon}||_{2}||\overline{\sigma}\Delta_{\delta}B^{H}||_{2}  \right|\right|_{L^2}$. We write
\begin{align}
\label{roughboundCS2}
\mathbb{E}\left[||\sigma_1\epsilon^{H}\Delta_{\delta}Y^{\epsilon}||_{2}^2||\overline{\sigma}\Delta_{\delta}B^{H}||_{2}^2\right] =& \epsilon^{2H}\overline{\sigma}^2\sigma_1^2\mathbb{E}\left[\left(\sum_{i=0}^{N-1}(Y^{\epsilon}_{i\delta, (i+1)\delta})^2\right)\left(\sum_{j=0}^{N-1}(B^{H}_{j\delta, (j+1)\delta})^2\right)\right]\nonumber\\ 
=&\epsilon^{2H}\overline{\sigma}^2\sigma_1^2\sum_{i,j=0}^{N-1}\mathbb{E}\left[(Y^{\epsilon}_{i\delta, (i+1)\delta})^2(B^{H}_{j\delta, (j+1)\delta})^2\right]\nonumber\\
\leq& \epsilon^{2H}\overline{\sigma}^2\sigma_1^2\sum_{i,j=0}^{N-1}\mathbb{E}[(Y^{\epsilon}_{i\delta, (i+1)\delta})^4]^{1/2}\mathbb{E}[(B^{H}_{j\delta, (j+1)\delta})^4]^{1/2}\nonumber\\
=&\epsilon^{2H}\overline{\sigma}^2\sigma_1^2 N^{2}\mathbb{E}[(Y^{\epsilon}_{0,\delta})^4]^{1/2}\mathbb{E}[(B^{H}_{0,\delta})^4]^{1/2}\nonumber\\
=& 3\epsilon^{2H}\overline{\sigma}^2\sigma_1^2 N^{2}\mathbb{E}[(Y^{\epsilon}_{0,\delta})^2]\mathbb{E}[(B^{H}_{0,\delta})^2]
= 3\epsilon^{2H}\overline{\sigma}^2\sigma_1^2 N^{2}\delta^{2H}.
\end{align}
It follows that
\begin{equation*}
\left| \left| ||\sigma_1\epsilon^{H}\Delta_{\delta}Y^{\epsilon}||_{2}||\overline{\sigma}^2\Delta_{\delta}B^{H}||_{2}  \right|\right|_{L^2} \lesssim N\epsilon^{H}\delta^{H}
\end{equation*}
which, again, together with lemma \ref{orderlemma}, yields
\begin{equation*}
\label{term2}
\frac{2\overline{\sigma}}{N}\left|\left|\left(\Delta_{\delta}B^{H}\right)^T P^{-1}\left(\Delta_{\delta}X^{\epsilon} - \overline{\sigma} \Delta_{\delta}B^{H}\right)\right|\right|_{L^2} \lesssim \epsilon^{H}/\delta^{\beta-H}
\end{equation*}
by just plugging each bound in \eqref{asbound2}. Finally, we control the third term of \eqref{eq: estimator decomposition}, using a Law of Large Numbers type of argument. Since $P$ is the covariance matrix of $\Delta_{\delta}B^{H}$ it holds that 
$$\left( \Delta_{\delta}B^{H}\right)^{T}\left( P\right)^{-1}\left( \Delta_{\delta}B^{H}\right) \sim \chi_{N}^2$$

Therefore, 
\begin{align*}
\left|\left|\frac{\overline{\sigma}^2}{N}\left( \Delta_{\delta}B^{H}\right)^{T}\left( P\right)^{-1}\left( \Delta_{\delta}B^{H}\right)-            \overline{\sigma}^2\right|\right|_{L^2} =&
\overline{\sigma}^2\left( \frac{1}{N^2}\mathbb{E}[(\chi_{N}^2)^2]- 2\frac{1}{N}\mathbb{E}[\chi_{N}^2] + 1\right)^{1/2} \nonumber\\
=& \overline{\sigma}^2\left(\frac{2}{N}\right)^{1/2} 
= \overline{\sigma}^{2}\delta^{1/2}\left(\frac{2}{T}\right)^{1/2}
\end{align*}
Putting everything together we get that, if $H>1/2$
\begin{equation}
\label{eq: order large}
\left|\left|\hat{\sigma}^2_{\epsilon,\delta} -\overline{\sigma}^2\right|\right|_{L^2} \lesssim \left(\epsilon/\delta \right)^{2H} + \left(\epsilon/\delta \right)^{H} + \delta^{1/2}
\end{equation}
while if $H<1/2$
\begin{equation}
\label{eq: order small}
\left|\left|\hat{\sigma}^2_{\epsilon,\delta} -\overline{\sigma}^2\right|\right|_{L^2} \lesssim \epsilon^{2H}/\delta + \epsilon^{H}/\delta^{1-H} + \delta^{1/2}.
\end{equation}
Therefore, if $H>1/2$ convergence is guaranteed provided that $\epsilon/\delta\longrightarrow 0$, whereas if $H<1/2$, we require that $\epsilon^{H}/\delta^{1-H}\longrightarrow 0$. In particular, if $\delta = \epsilon^{\alpha}$ for some $\alpha>0$, convergence will be guaranteed for $0<\alpha<\min\left\{1,\frac{H}{1-H}\right\}$.
\end{proof}
\section{Some considerations on the practical usage of $\hat{\sigma}_{\epsilon,\delta}$}
Until now, we have assumed that $H$ is known. In most practical situations, this will be an unrealistic assumption. Building on the results in \cite{discretetimeinf}, we propose below an estimator for $H$.

\begin{proposition}
Let $\hat{H}_{\delta,\epsilon}$ be defined as
\begin{equation}
\label{eq: estimator of H}
\hat{H}_{\delta,\epsilon} = \frac{1}{2} - \frac{1}{2\ln(2)}\ln\left(\frac{\sum_{k=2}^{2n}\left(\Delta_{k,\delta/2}^{(2)}X^{\epsilon} \right)^2}{\sum_{k=2}^{n}\left(\Delta_{k,\delta}^{(2)}X^{\epsilon} \right)^2} \right) 
\end{equation}
where $\Delta^{(2)}_{k,\delta}X = X_{k\delta}-2X_{(k-1)\delta}+X_{(k-2)\delta}$ represents the second-order differences. Assuming that $
\delta(\epsilon)$ is such that $\delta(\epsilon)\rightarrow0$ and $\epsilon/\delta\rightarrow 0$ as $\epsilon\rightarrow 0$. Then, 
$$\hat{H}_{\delta,\epsilon}\xrightarrow{\epsilon\rightarrow0^+} H \text{ in probability.}$$
\end{proposition}
\begin{proof}
We first derive the estimate 
\begin{equation}
\label{eq: estimate on difference of sums of squares second order differences}
n^{2H-1}\left|\sum_{k=2}^n\left(\Delta_{k,\delta}^{(2)}X^{\epsilon} \right)^2 -\sum_{k=2}^n\left(\overline{\sigma}\Delta_{k,\delta}^{(2)}B^H  \right)^2 \right|\lesssim \left(\epsilon/\delta\right)^{2H} + \left(\epsilon/\delta\right)^{H}
\end{equation}
in $L^2(\Omega)$ by using the fact that
\begin{align*}
\left(\Delta_{k,\delta}^{(2)}X^{\epsilon} \right)^2 -\left(\overline{\sigma}\Delta_{k,\delta}^{(2)}B^H  \right)^2 =&\left(\Delta_{k,\delta}^{(2)}X^{\epsilon}  -\overline{\sigma}\Delta_{k,\delta}^{(2)}B^H  \right)^2\\ &+ 2\left(\Delta_{k,\delta}^{(2)}X^{\epsilon}  -\overline{\sigma}\Delta_{k,\delta}^{(2)}B^H  \right)\left( \overline{\sigma}\Delta_{k,\delta}^{(2)}B^H \right).
\end{align*}
Recall that $X_{s,t}^{\epsilon} -\overline{\sigma} B_{s,t}^{H} =\epsilon^H\sigma_1 Y_{s,t}^{\epsilon}$ for any values of $s,t$. It follows that 
$$\Delta_{k,\delta}^{(2)}X^{\epsilon}  -\overline{\sigma}\Delta_{k,\delta}^{(2)}B^H  = \epsilon^H \sigma_1\Delta_{k,\delta}^{(2)}Y^{\epsilon}$$
Then, under the assumption of stationarity of $Y^{\epsilon}$, a direct applications of Cauchy-Schwardz's inequality leads to the following the estimates 
\begin{align*}
\epsilon^{2H} \sigma_1^2\left|\left|(\Delta_{k,\delta}^{(2)}Y^{\epsilon})^2\right|\right|_{L^2}&\lesssim \epsilon^{2H}\\
\epsilon^{H}\overline{\sigma}\sigma_1 \left|\left|(\Delta_{k,\delta}^{(2)}Y^{\epsilon})(\Delta_{k,\delta}^{(2)}B^{H})\right|\right|_{L^2}&\lesssim \epsilon^{H}\delta^{H}
\end{align*}
With these estimates, the result can be obtained following the same strategy as in \cite{discretetimeinf}.
\end{proof}
In the supplement, we provide some numerical evidence of the behavior of this estimator and how the interaction of $\epsilon$ and $\delta$ determines its performance.

A natural concern is the effect of a poorly estimated $H$ on the estimation of $\overline{\sigma}$. Although relatively large departures in the estimation of $H$ have an impact on the estimation, we also provide experimental evidence that the estimator is continuous in $H$, in some sense. Small deviations in $H$ lead to small fluctuations in the estimation, but they do not trigger unexpected blow-up or vanishing behaviors. 

\begin{acks}
We would like to thank the anonymous referees for the very useful comments and suggestions. HB and AP acknowledge support by EPSRC grant EP/W00707X/1. PRAM acknowledges funding from the Mathematics and Statistics CDT at University of Warwick. 
\end{acks}
\appendix
\section{Supplementary Material: Bounds for $A_i$ and $B_{j_2}$ in Proof 3.3.}
\label{sec: details of calculations}
We include here the details of how the bounds for the different terms of the $L^2(\mathbb{R}^+)$ norm of $\mathcal{D}_{-}^{H-1/2}(f)$ are calculated in the proof of lemma 3.3 when $f=\sum_{i=0}^{n-1}u_i\mathbf{1}_{(i\delta, (i+1)\delta]}$. 
\subsection{Bound for $A_i$}
We first deal with the terms $j_1,j_2\geq i+2$: 
\begin{align*}
&\sum_{i=0}^{n-1}\sum_{j_1,j_2\geq i+2}u_i^2\int_i^{i+1}\left((j_1+1-s)^{-\gamma} - (j_1-s)^{-\gamma}\right)\left((j_2+1-s)^{-\gamma} - (j_2-s)^{-\gamma}\right)ds
\\
\leq&\sum_{i=0}^{n-1}\sum_{j_1,j_2\geq i+2}u_i^2\int_i^{i+1}\gamma^2(j_1-s)^{-\gamma-1}(j_2-s)^{-\gamma-1}ds \\
\leq&\gamma^2 \sum_{i=0}^{n-1}u_i^2\sum_{j_1,j_2\geq i+2}\gamma^2(j_1-i-1)^{-\gamma-1}(j_2-i-1)^{-\gamma-1}
= \gamma^2\sum_{i=0}^{n-1}u_i^2\left(\sum_{k=1}^{n-(i+2)}k^{-\gamma-1}\right)^2\\
<&\gamma^2\sum_{i=0}^{n-1}u_i^2\left(\sum_{k=1}^{\infty}k^{-\gamma-1}\right)^2 = C\sum_{i=0}^{n-1}u_i^2
\end{align*}
where the first inequality is due to the Mean Value Theorem applied to both terms of the integrand and the last equality holds since $-\gamma-1 = -H-1/2<-1$. Note that within this term the role of $j_1$ and $j_2$ is still interchangeable. Therefore, we only need to analyse now the terms corresponding to $j_2>j_1=i+1$ and the term $j_2 = j_1=i+1$, since the terms $j_1\geq j_2$ are the same. 
\begin{align}
&\sum_{i=0}^{n-1}u_i^2\sum_{j_1 = i+2}^{n-1}\int_{i}^{i+1}\left((j_1+1-s)^{-\gamma} - (j_1-s)^{-\gamma}\right) \left((i+2-s)^{-\gamma} - (i+1-s)^{-\gamma}\right)ds \nonumber\\
\leq&\sum_{i=0}^{n-1}u_i^2\sum_{j_1=i+2}^{n-1}\int_{i}^{i+1}-\gamma (j_1-s)^{-\gamma-1}\left((i+2-s)^{-\gamma} - (i+1-s)^{-\gamma}\right)ds 
\nonumber\\
\leq & \sum_{i=0}^{n-1}u_i^2\int_{i}^{i+1}-\gamma \left((i+2-s)^{-\gamma} - (i+1-s)^{-\gamma}\right)
\left(\sum_{j_1=i+2}^{n-1}(j_1-i-1)^{-\gamma-1}\right)ds 
\label{line: geometric sum} \\
< &\sum_{i=0}^{n-1}u_i^2\int_{0}^{1}-\gamma \left((2-r)^{-\gamma} - (1-r)^{-\gamma}\right)\left(\sum_{k=1}^{n-(i+2)}k^{-\gamma-1}\right)ds \nonumber\\
= & C\sum_{i=0}^{n-1}u_i^2 \left[\frac{(2-r)^{-\gamma+1}}{1-\gamma}+\frac{(1-r)^{-\gamma+1}}{1-\gamma}\right]_{r=0}^{r=1}
= C\sum_{i=0}^{n-1}u_i^2 \nonumber
\end{align}
where the geometric sum in \eqref{line: geometric sum} is bounded as before and the integrand integrates to a constant as $0<\gamma<1/2$.

We now bound the terms corresponding to $j_1=j_2=i+2$:
\begin{align*}
&\sum_{i=0}^{n-1}u_i^2\int_{i}^{i+1}\left((i+2-s)^{-\gamma} - (i+1-s)^{-\gamma}\right)^2ds \\
&\leq 2\sum_{i=0}^{n-1}u_i^2\int_0^1\left((2-r)^{-2\gamma} + (1-r)^{-2\gamma}\right)dr=C\sum_{i=0}^{n-1}u_i^2
\end{align*}
where again the integral is just a constant as $0<2\gamma<1$.
\subsection{Bound for $B_{j_2}$}
We use similar techniques to the ones used for $A_i$, noting that more care should be taken as the role of $j_1$ and $j_2$ will no longer be interchangeable as $j_2$ appears multiplying $I(i, j_1, j_2)$. We first split the sum as follows: 
\begin{align*}
\sum_{i=0}^{n-1}\sum_{j_1,j_2\geq i+1}u_{j_2}^2I(i,j_1,j_2)=\sum_{i=0}^{n-1}\sum_{j_1\geq j_2\geq i+1}u_{j_2}^2I(i,j_1,j_2)+ \sum_{i=0}^{n-1}\sum_{j_2>j_1\geq i+1}u_{j_2}^2I(i,j_1,j_2)
\end{align*}
and we bound each term separately. For the first one we can proceed as follows:
\begin{align*}
&\sum_{i=0}^{n-1}\sum_{j_1\geq j_2\geq i+1}u_{j_2}^2I(i,j_1,j_2) = \sum_{j_1\geq j_2\geq 1}\sum_{i=0}^{j_2-1} u_{j_2}^2I(i,j_1,j_2)\\
=&\sum_{j_2\geq j_1\geq 1}u_{j_2}^2 \int_0^{j_2}\left((j_1+1-s)^{-\gamma} - (j_1-s)^{-\gamma}\right)\left((j_2+1-s)^{-\gamma} - (j_2-s)^{-\gamma}\right)ds \\
=& \sum_{j_2=1}^{n-1}u_{j_2}^2\int_{0}^{j_2}\sum_{j_1=j_2}^{n-1}\left((j_1+1-s)^{-\gamma} - (j_1-s)^{-\gamma}\right)\left((j_2+1-s)^{-\gamma} - (j_2-s)^{-\gamma}\right)ds\\
=&\sum_{j_2=1}^{n-1}u_{j_2}^2\int_0^{j_2}\left((n-s)^{-\gamma} - (j_2-s)^{-\gamma}\right)\left((j_2+1-s)^{-\gamma} - (j_2-s)^{-\gamma}\right)ds
\end{align*}
Note that, for $s$ in $[0,j_2)$, it holds that $$\left((n-s)^{-\gamma} - (j_2-s)^{-\gamma}\right)\left((j_2+1-s)^{-\gamma} - (j_2-s)^{-\gamma}\right)\leq 2(j_2-s)^{-2\gamma}$$
Therefore, 
\begin{align*}
&\sum_{j_2=1}^{n-1}u_{j_2}^2\int_0^{j_2}\left((n-s)^{-\gamma} - (j_2-s)^{-\gamma}\right)\left((j_2+1-s)^{-\gamma} - (j_2-s)^{-\gamma}\right)ds \\
&\leq \sum_{j_2=1}^{n-1}u_{j_2}^2\int_0^{j_2}2(j_2-s)^{-2\gamma}ds = 2\sum_{j_1=1}^{n-1}u_{j_2}^2\left[\frac{(j_2-s)^{-2\gamma+1}}{-2\gamma + 1} \right]_0^{j_2}\leq C\sum_{j_2=0}^{n-1}u_{j_2}^2.
\end{align*}
For the remaining terms, we use the following procedure:
\begin{align*}
   &\sum_{i=0}^{n-1}\sum_{j_2>j_1\geq i+1}u_{j_2}^2I(i,j_1,j_2)\\ =&
   \sum_{i=0}^{n-1}\sum_{j_2 = i+2}^{n-1}u_{j_2}^2\int_{i}^{i+1}\sum_{j_1 = i+1}^{j_2-1}\left((j_1+1-s)^{-\gamma} -(j_1-s)^{-\gamma}\right) \left((j_2+1-s)^{-\gamma} - (j_2-s)^{-\gamma}\right)ds\\
   =&\sum_{i=0}^{n-1}\sum_{j_2 = i+2}^{n-1}u_{j_2}^2\int_{i}^{i+1}\left((j_2-s)^{-\gamma} - (i+1-s)^{-\gamma}\right)
   \left((j_2+1-s)^{-\gamma} - (j_2-s)^{-\gamma}\right)ds\\
   \leq&\sum_{i=0}^{n-1}\sum_{j_2 = i+2}^{n-1}u_{j_2}^2\int_{i}^{i+1} - (i+1-s)^{-\gamma}\left((j_2+1-s)^{-\gamma} - (j_2-s)^{-\gamma}\right)ds\\
   \leq& \gamma\sum_{j_2 = 2}^{n-1}u_{j_2}^2\sum_{i=0}^{j_2-2}\int_{i}^{i+1}(i+1-s)^{-\gamma}(j_2-s)^{-\gamma-1}ds\\
   \leq& \gamma\sum_{j_2 = 2}^{n-1}u_{j_2}^2\sum_{i=0}^{j_2-2}(j_2-(i+1))^{-\gamma-1}\int_{i}^{i+1}(i+1-s)^{-\gamma}ds \\
   =& \frac{\gamma}{1-\gamma}\sum_{j_2}u_{j_2}^{2}\sum_{k=1}^{j_2-1}k^{-\gamma-1} < C\sum_{j_2=0}^{n-1}u_{j_2}^{2}
\end{align*}
\section{Supplementary Numerical Experiments.}
\subsection{Estimation of $H$}
We illustrate here some results of the estimation of $H$ when the true values were $0.3, 0.5$ and $0.7$. We set $\overline{\sigma}=1$ for all cases and we simulate sample paths of $X_t^{\epsilon}$ using an Euler-Maruyama scheme and a standard rectangle rule for numerical integration with step size $\Delta_t = \frac{1}{2}10^{-4}$ and $100$ sample paths are simulated for each parameter combination to compute sample means. The time horizon is fixed to $T=10$ in all cases to ensure each subsampled set gets enough points.

The results for different values of $\alpha$ emphasize the need to minimize $\delta$ at the same time as the ratio $\epsilon/\delta$ which is indeed a trade-off for a fixed realization. For a given sample path, when $\alpha$ is small the step size in the subsample has gone too wide leading to poor estimations, especially in terms of the variance. On the other hand, when $\alpha$ grows too much the compared to $\epsilon$ the convergence of the estimator gets slower due to $X^{\epsilon}$ not being close enough to $B^H$ in terms of squared increments as seen in (7.2). Values in the middle instead lead to good estimation in terms of both bias and variance, usually around $\alpha=0.5$ as seen below. 
\begin{table}[H]
\centering
\begin{tabular}{|l| |c|c|c|c|c|c|}
\hline
$\epsilon \backslash \alpha$     & 0.1 & 0.3 & 0.5 & 2/3 & 0.9 & 1 \\
\hline
\hline
0.01              & 0.16830 & 0.31319 & 0.45312 & 0.64869 & 0.9964 &  1.10033\\
\hline
0.001             & 0.23840 & 0.28555 & 0.34747 & 0.46094 & 0.94598 & 1.08766\\
\hline
0.0001            & 0.23301 & 0.30405 & 0.31743 & 0.35108 & 0.70847 & 0.73053\\
\hline

\end{tabular}
\label{table: bla bla}
\caption{Means for $H=0.3$.}
\end{table}

\begin{table}[H]
\centering
\begin{tabular}{|l| |c|c|c|c|c|c|}
\hline
$\epsilon \backslash \alpha$     & 0.1 & 0.3 & 0.5 & 2/3 & 0.9 & 1 \\
\hline
\hline
0.01              & 0.41868 & 0.50932 & 0.62304 & 0.82310 & 1.15433 & 1.26152\\
\hline
0.001             &0.35710 & 0.51826 & 0.53115 & 0.63189 & 1.11434 & 1.25519\\
\hline
0.0001            &0.44367 & 0.48776 & 0.51009 & 0.53981 & 0.99066 & 0.95352\\
\hline

\end{tabular}
\label{table: bla bla}
\caption{Means for $H=0.5$.}
\end{table}
\begin{table}[H]
\centering
\begin{tabular}{|l| |c|c|c|c|c|c|}
\hline
$\epsilon \backslash \alpha$     & 0.1 & 0.3 & 0.5 & 2/3 & 0.9 & 1 \\
\hline
\hline
0.01              & 0.68404 & 0.65385 & 0.80445 & 0.98293 & 1.30505 & 1.41248\\
\hline
0.001             & 0.61798 & 0.67400 &  0.71797 & 0.81158 & 1.27097 & 1.40650\\
\hline
0.0001            & 0.64902 & 0.67946 & 0.69738 & 0.73340 & 1.23520 & 1.15784\\
\hline

\end{tabular}
\label{table: bla bla}
\caption{Means for $H=0.7$.}
\end{table}

\begin{table}[H]
\centering
\begin{tabular}{|l| |c|c|c|c|c|c|}
\hline
$\epsilon \backslash \alpha$     & 0.1 & 0.3 & 0.5 & 2/3 & 0.9 & 1 \\
\hline
\hline
0.01              & 0.51159 & 0.24122 & 0.10433 & 0.08539 & 0.04016 & 0.02777\\
\hline
0.001             & 0.38315 & 0.15750 & 0.07712 & 0.04010 & 0.01562 & 0.00894\\
\hline
0.0001            & 0.28052 & 0.12071 & 0.04378 & 0.02162 & 0.00686 & 0.00350\\
\hline

\end{tabular}
\label{table: bla bla}
\caption{Standard deviations for $H=0.3$.}
\end{table}
\begin{table}[H]
\centering
\begin{tabular}{|l| |c|c|c|c|c|c|}
\hline
$\epsilon \backslash \alpha$     & 0.1 & 0.3 & 0.5 & 2/3 & 0.9 & 1 \\
\hline
\hline
0.01              & 0.41286 & 0.24802 & 0.11882 & 0.07566 & 0.03306 & 0.02376\\
\hline
0.001             & 0.33681 & 0.13741 & 0.07351 & 0.03770 & 0.01375 & 0.00694\\
\hline
0.0001            & 0.27090 & 0.10555 & 0.03600 & 0.02176 & 0.00621 & 0.00320\\
\hline

\end{tabular}
\label{table: bla bla}
\caption{Standard deviations for $H=0.5$.}
\end{table}
\begin{table}[H]
\centering
\begin{tabular}{|l| |c|c|c|c|c|c|}
\hline
$\epsilon \backslash \alpha$     & 0.1 & 0.3 & 0.5 & 2/3 & 0.9 & 1 \\
\hline
\hline
0.01              & 0.33537 & 0.19520 & 0.11140 & 0.06598 & 0.02655 & 0.02046\\
\hline
0.001             & 0.28493 & 0.13076 & 0.06440 & 0.03185 & 0.01129 & 0.00710\\
\hline
0.0001            & 0.23104 & 0.08949 & 0.03452 & 0.01508 & 0.00503 & 0.00247\\
\hline
\end{tabular}
\label{table: bla bla}
\caption{Standard deviations for $H=0.7$.}
\end{table}
\subsection{Robustness of $\hat{\sigma}_{\delta,\epsilon}^2$ with respect to $H$.}
In this section, we include a numerical simulation of the robustness of the plug-in MLE estimator against a mispecification of $H$. Empirical means are computed in the exact same way as in the previous section, setting $\overline{\sigma}=1$ for all cases too. We artificially contaminate $H$ by using instead $\hat{H}=H+\eta$ where $\eta$ is taken with random sign and decreasing magnitude. We demonstrate empirically that the estimator is continuous with respect to $H$ in the sense that small deviations in $H$ lead to small departures in the estimation. We measure this by comparing the estimator with a contaminated choice of $H$ with the the value obtained when using the true $H$ for each realization of a sample path. 

We perform this experiment for $H=0.3, 0.5$ and $0.7$. To ensure convergence of the true estimator we take $\epsilon=0.0001$ and subsample from the generated paths with rate $\delta = \epsilon^{\alpha}$ for $\alpha = 0.42, 2/3$ and $0.58$ respectively. For $H=1/2$ we set $T=1$ whereas for the other two cases we enlarge this to $T=10$ to make sure there is enough data. 

\begin{table}[H]
\centering
\begin{tabular}{|c| |c|c|c|c|c|c|}
\hline
$|\eta|$     & 0.03 & 0.01 & 0.05 & 0.001\\
\hline
\hline
$\mathbb{E}[\hat{\sigma}_{\tilde{H}}^2]$          & 0.96666 & 0.97665 & 0.98217 & 0.97741\\
\hline
$||\hat{\sigma}_{\tilde{H}}^2 - \hat{\sigma}_{H}^2||_{L^2}  $             & 0.19908 & 0.06682 & 0.03366 & 0.00669\\
\hline
\end{tabular}
\label{table: bla bla}
\caption{$H=0.3$.}
\end{table}

\begin{table}[H]
\centering
\begin{tabular}{|c| |c|c|c|c|c|c|}
\hline
$|\eta|$     & 0.03 & 0.01 & 0.05 & 0.001\\
\hline
\hline
$\mathbb{E}[\hat{\sigma}_{\tilde{H}}^2]$          & 1.02740 & 0.96911 & 0.96971 & 0.96867\\
\hline
$||\hat{\sigma}_{\tilde{H}}^2 - \hat{\sigma}_{H}^2||_{L^2}  $             & 0.36978 & 0.11841 & 0.05929 & 0.01188\\
\hline
\end{tabular}
\label{table: bla bla}
\caption{$H=0.5$.}
\end{table}

\begin{table}[H]
\centering
\begin{tabular}{|c| |c|c|c|c|c|c|}
\hline
$|\eta|$     & 0.03 & 0.01 & 0.05 & 0.001\\
\hline
\hline
$\mathbb{E}[\hat{\sigma}_{\tilde{H}}^2]$          & 1.09217 & 1.00776 & 0.99679 & 0.99127\\
\hline
$||\hat{\sigma}_{\tilde{H}}^2 - \hat{\sigma}_{H}^2||_{L^2}  $             & 0.38860 & 0.12298 & 0.06105 & 0.01215\\
\hline
\end{tabular}
\label{table: bla bla}
\caption{$H=0.7$.}
\end{table}
\end{document}